\documentclass[12pt]{amsart}

\usepackage[english]{babel}

\usepackage[pdftex,textwidth=400pt,marginratio=1:1]{geometry}

\usepackage{amsfonts}

\usepackage[dvips]{graphics}

\usepackage[colorinlistoftodos]{todonotes}

\usepackage{amsmath}

\usepackage{amsthm}

\usepackage{amssymb}

\usepackage{bbm}

\usepackage{cancel}

\usepackage{color}

\usepackage[curve]{xypic}

\usepackage{graphicx}

\newtheorem{theorem}{Theorem}[section]

\newtheorem{lemma}[theorem]{Lemma}

\theoremstyle{definition}

\newtheorem{remark}[theorem]{Remark}

\theoremstyle{property}

\makeatletter
\newcommand{\contraction}[5][1ex]{%
  \mathchoice
    {\contraction@\displaystyle{#2}{#3}{#4}{#5}{#1}}%
    {\contraction@\textstyle{#2}{#3}{#4}{#5}{#1}}%
    {\contraction@\scriptstyle{#2}{#3}{#4}{#5}{#1}}%
    {\contraction@\scriptscriptstyle{#2}{#3}{#4}{#5}{#1}}}%
\newcommand{\contraction@}[6]{%
  \setbox0=\hbox{$#1#2$}%
  \setbox2=\hbox{$#1#3$}%
  \setbox4=\hbox{$#1#4$}%
  \setbox6=\hbox{$#1#5$}%
  \dimen0=\wd2%
  \advance\dimen0 by \wd6%
  \divide\dimen0 by 2%
  \advance\dimen0 by \wd4%
  \vbox{%
    \hbox to 0pt{%
      \kern \wd0%
      \kern 0.5\wd2%
      \contraction@@{\dimen0}{#6}%
      \hss}%
    \vskip 0.5ex
    \vskip\ht2}}

\newcommand{\contraction@@}[3][0.05em]{%
  \hbox{%
    \vrule width #1 height 0pt depth #3%
    \vrule width #2 height 0pt depth #1%
    \vrule width #1 height 0pt depth #3%
    \relax}}
\makeatother

\DeclareFontFamily{OT1}{rsfs}{}
\DeclareFontShape{OT1}{rsfs}{n}{it}{<-> rsfs10}{}
\DeclareMathAlphabet{\curly}{OT1}{rsfs}{n}{it}

\newcommand\Sing{\mathrm{Sing}}

\newcommand\Quot{\mathrm{Quot}}

\newcommand\onto{\mathrm{onto}}

\newcommand\pure{\mathrm{pure}}

\renewcommand\S{\mathcal S}
\renewcommand\O{\mathcal O}
\newcommand\PP{\mathbb P}

\newcommand\mdot{{\scriptscriptstyle\bullet}}

\newcommand\cE{\mathcal E}

\newcommand\F{\mathcal F}

\newcommand\bfone{\mathbf{1}}
\newcommand\bftwo{\mathbf{2}}
\newcommand\bfthree{\mathbf{3}}

\newcommand\T{\mathcal T}

\newcommand\C{\mathbb C}

\newcommand\Q{\mathbb Q}
\newcommand\cQ{\mathcal Q}

\newcommand\ccR{\mathcal R}

\newcommand\Z{\mathbb Z}

\newcommand\INTO{\ar@{^{(}->}[r]}

\newcommand\rk{\operatorname{rk}}

\newcommand\coker{\operatorname{coker}}

\newcommand\Hom{\operatorname{Hom}}
\renewcommand\hom{\curly H\!om}
\newcommand\Ext{\operatorname{Ext}}
\newcommand\ext{\curly Ext}

\newcommand\Hilb{\operatorname{Hilb}}

\newcommand\beq[1]{\begin{equation}\label{#1}}
\newcommand\eeq{\end{equation}}
\newcommand\beqa{\begin{eqnarray*}}
\newcommand\eeqa{\end{eqnarray*}}

\DeclareRobustCommand{\SkipTocEntry}[4]{}

\begin{document}
\title[Rank 2 wall-crossing and the Serre correspondence]{Rank 2 wall-crossing and the Serre correspondence}
\author[A.~Gholampour and M.~Kool]{Amin Gholampour and Martijn Kool}
\maketitle

\begin{abstract}
We study Quot schemes of 0-dimensional quotients of sheaves on 3-folds $X$. When the sheaf $\ccR$ is rank 2 and reflexive, we prove that the generating function of Euler characteristics of these Quot schemes is a power of the MacMahon function times a polynomial. This polynomial is itself the generating function of Euler characteristics of Quot schemes of a certain 0-dimensional sheaf, which is supported on the locus where $\ccR$ is not locally free. 

In the case $X = \C^3$ and $\ccR$ is equivariant, we use our result to prove an explicit product formula for the generating function. This formula was first found using localization techniques in previous joint work with B.~Young. Our results follow from R.~Hartshorne's Serre correspondence and a rank 2 version of a Hall algebra calculation by J.~Stoppa and R.~P.~Thomas.
\end{abstract}
\thispagestyle{empty}


\section{Introduction} \label{intro}

Let $\ccR$ be a reflexive\footnote{A torsion free sheaf $\F$ is called reflexive if the natural embedding $\F \hookrightarrow \F^{**}$ is an isomorphism. Here $(\cdot)^* := \hom(\cdot, \O_X)$ denotes the dual.} sheaf on a smooth 3-fold $X$. Since $X$ is 3-dimensional, $\ccR$ is locally free outside a 0-dimensional closed subset. For any torsion free sheaf $\F$ on $X$ its double dual $\F^{**}$ is reflexive and we have a natural embedding $\F \hookrightarrow \F^{**}$. Reflexive sheaves are much easier than torsion free sheaves, because they are determined by their restriction to the complement of any closed subset of codimension $\geq 2$ \cite[Prop.~1.6]{Har2}. 

Let $X = \C^3$ and consider the natural action of $T = \C^{*3}$ on $X$. Suppose $\ccR$ is a rank 2 $T$-equivariant reflexive sheaf on $X$. Then $H^0(\ccR)$ is graded by the character group $X(T) = \Z^3$. We denote by $\Quot(\ccR,n)$ the Quot scheme of 0-dimensional quotients $\ccR \twoheadrightarrow \cQ$,
where $\cQ$ has length $n$. Using $T$-localization, the authors and B.~Young calculated the topological Euler characteristics $e\big(\Quot(\ccR,n)\big)$ \cite{GKY1}. The answer is as follows. 

In the case $\ccR$ is locally free, it splits $T$-equivariantly as a sum of two line bundles \cite[Prop.~3.4]{GKY1}. The splitting gives an additional scaling action of $\C^*$ on the summands. It is not hard to show that localization with respect to $T \times \C^*$ gives
$$
\sum_{n=0}^{\infty} e\big(\Quot(\ccR,n)\big) \, q^n = M(q)^2,
$$
where $M(q) = \prod_{k>0} 1/(1-q^k)^k$ denotes the MacMahon function counting 3D partitions. The interesting case is when $\ccR$ is singular, i.e.~not locally free. In this case the rank 2 module $H^0(\ccR)$ has \emph{three} homogeneous generators, whose weights can be uniquely written as  
\begin{equation} \label{wghts}
(u_1,u_2,u_3) + (v_1,v_2,0), \ (u_1,u_2,u_3) + (v_1,0,v_3), \ (u_1,u_2,u_3) + (0,v_2,v_3),
\end{equation}
where $u_1,u_2,u_3 \in \Z$ and $v_1,v_2,v_3 \in \Z_{>0}$. See \cite[Prop.~3.4]{GKY1} for details.
\begin{theorem} [Gholampour-Kool-Young] \label{GKYthm}
Let $\ccR$ be a singular rank 2 $T$-equivariant reflexive sheaf on $\C^3$ with homogeneous generators of weights \eqref{wghts}. Then
$$
\sum_{n=0}^{\infty} e\big(\Quot(\ccR,n)\big) \, q^n = M(q)^2 \prod_{i=1}^{v_1} \prod_{j=1}^{v_2} \prod_{k=1}^{v_3} \frac{1-q^{i+j+k-1}}{1-q^{i+j+k-2}}.
$$

\end{theorem}
This formula was used as a building block for generating functions of Euler characteristics and virtual invariants of moduli spaces of rank 2 stable torsion free sheaves on toric 3-folds in \cite{GKY1}. The RHS of this formula is $M(q)^2$ times the generating function of 3D partitions confined to the box $[0,v_1] \times [0,v_2] \times [0,v_3]$ \cite[(7.109)]{Sta}.\footnote{More precisely: for each box of the 3D partition, the corner closest to the origin is an element of $[0,v_1-1] \times [0,v_2-1] \times [0,v_3-1]$.} The components of the fixed locus $\Quot(\ccR,n)^T$ are unions of products of $\PP^1$ and the combinatorics of their enumeration is solved using the double dimer model in \cite{GKY2}. 
From the shape of the formula, we expected another proof by wall-crossing should exist, which prompted this paper. In this paper, we derive Theorem \ref{GKYthm} from Theorem \ref{main} below.

Let $X$ be any smooth projective 3-fold and $\ccR$ a rank 2 reflexive sheaf on $X$. Then $\ext^1(\ccR,\O_X)$ is a 0-dimensional sheaf supported on the singularities of $\ccR$, i.e.~on the locus where $\ccR$ is \emph{not} locally free. The length of this sheaf is equal to $c_3(\ccR)$ \cite[Prop.~2.6]{Har2}. 
\begin{theorem} \label{main}
Let $\ccR$ be a rank 2 reflexive sheaf on a smooth projective 3-fold $X$. Suppose $H^1(\det(\ccR))=H^2(\det(\ccR)) = 0$ and there exists a cosection $\ccR \rightarrow \O_X$ cutting out a 1-dimensional closed subscheme. Then
\begin{equation} \label{maineq}
\sum_{n=0}^{\infty} e\big(\Quot(\ccR,n)\big) \, q^n = M(q)^{2e(X)} \sum_{n=0}^{\infty} e\big(\Quot(\ext^1(\ccR,\O_X),n)\big) \, q^n.
\end{equation}
In particular the RHS is $M(q)^{2e(X)}$ times a polynomial of degree $c_3(\ccR)$.
\end{theorem} 

\begin{remark} \label{weakassump}
In an earlier version of this paper, we conjectured that \eqref{maineq} holds for any rank 2 reflexive sheaf $\ccR$ on any smooth projective threefold $X$. At the end of this introduction, we present an argument due to J.~Rennemo, which indeed shows that this is the case. Essentially, after replacing $\ccR$ by an appropriate twist $\ccR(-mH)$, the vanishing conditions are automatically satisfied and the desired cosection always exists. Alternatively, in the sequel \cite{GK2}, we generalize \eqref{maineq} to arbitrary torsion free sheaves of homological dimension $\leq 1$ and any rank on any smooth projective threefold. 
\end{remark}

\begin{remark} [Hilbert scheme version] \label{Hilbremark}
Let $\ccR$ be a rank 2 reflexive sheaf on a smooth projective 3-fold $X$. Since $\ccR$ is rank 2 and reflexive, it has homological dimension $\leq 1$ and admits a 2-term resolution by vector bundles \cite[Prop.~1.3]{Har2}
$$
0 \rightarrow \cE_1 \rightarrow \cE_0 \rightarrow \ccR \rightarrow 0.
$$ 
A priori we only know $\rk \cE_0 - \rk \cE_1 = 2$. However in the \emph{minimal rank case}, i.e.~$\rk \cE_1 = 1$, we can dualize the above sequence and obtain
$$
\cdots \rightarrow \cE_{1}^{*} \twoheadrightarrow \ext^1(\ccR,\O_X),
$$
which shows $\ext^1(\ccR,\O_X)$ is a structure sheaf (recall that $\ext^1(\ccR,\O_X)$ is 0-dimensional). We denote the corresponding 0-dimensional subscheme by $\Sing(\ccR)$. 
Then \eqref{maineq} implies
\begin{equation} \label{Hilb}
\sum_{n=0}^{\infty} e\big(\Quot(\ccR,n)\big) \, q^n = M(q)^{2e(X)} \sum_{n=0}^{\infty} e\big(\Hilb^n(\Sing(\ccR)) \big) \, q^n,
\end{equation}
where $\Hilb^n(\Sing(\ccR))$ denotes the Hilbert scheme of $n$ points on $\Sing(\ccR)$. In particular the RHS is $M(q)^{2e(X)}$ times a polynomial of degree $\ell(\Sing(\ccR))$, where $\ell(\cdot)$ denotes length. 
Theorem \ref{GKYthm} is derived from \eqref{Hilb} combined with a $T$-localization argument in Section \ref{torus}.
\end{remark}

\subsection{Idea of proof and heuristics} The proof of Theorem \ref{main} uses a rank 2 version of a wall-crossing (or rather Hall algebra) calculation by J.~Stoppa and R.~P.~Thomas \cite{ST} along the lines of \cite[Section 3.3]{PT1}. These Hall algebra methods were introduced by D.~Joyce \cite{Joy1,Joy2, Joy3, Joy4}, and M.~Kontsevich and Y.~Soibelman \cite{KS}, see also \cite{Bri1, Bri2}. We provide the main ideas leading to our formula:

\begin{enumerate}
\item On one side of the wall (DT side), we are concerned with \emph{surjections}\footnote{Surjectivity can be regarded as a stability condition on the DT side.} inside
$$
\Ext^1(\ccR,\cQ[-1]) \cong \Hom(\ccR,\cQ),
$$ corresponding to short exact sequences $$0\to \F\to \ccR\to \cQ\to 0,$$
where $\ccR$ is fixed and $\cQ$ varies among all 0-dimensional sheaves on $X$. Equivalenty, $\F$ varies over all rank 2 torsion free sheaves with reflexive hull $\ccR$ and 0-dimensional cokernel. The sequence above gives the exact triangle $$\cQ[-1]\to \F\to \ccR$$ in the derived category $D^b(X)$. This is a rank 2 analog of (3.6) in \cite{PT1}. 
\item On the other side of the wall (PT side), we obtain objects $A^\bullet \in D^b(X)$ given by 2-term complexes of coherent sheaves that fit in exact triangles of the form $$\ccR\to A^\bullet \to \cQ[-1],$$ i.e.~elements of $\Ext^1(\cQ[-1],\ccR)$. This exact triangle is an analog of (3.7) in \cite{PT1}. Next we use our assumption: the existence of a cosection $\ccR \rightarrow \O_X$ cutting out a 1-dimensional closed subscheme $C \subset X$. This enables us to describe the objects $A^\bullet$ which are ``stable'' on the PT side of the wall. The cosection allows us to embed (Lemma \ref{lem2})
$$
\Ext^1(\cQ[-1],\ccR) \subset \Ext^1(\cQ[-1],I_C) \cong \Ext^1(\cQ,\O_C),
$$
where $I_C \subset \O_X$ denotes the ideal sheaf of $C \subset X$ and the second isomorphism follows from the usual short exact sequence. Note that elements of $\Ext^1(\cQ,\O_C)$ correspond to short exact sequences $$0 \rightarrow \O_C \rightarrow F \rightarrow \cQ \rightarrow 0.$$ The elements of $\Ext^1(\cQ,\O_C)$ with \emph{pure}\footnote{Purity can be regarded as a stability condition on the PT side.} $F$ are exactly the stable pairs (PT pairs) with underlying Cohen-Macaulay curve $C$. Therefore, inside $\Ext^1(\cQ[-1],\ccR)$ we have the locus of elements for which the corresponding extension in $\Ext^1(\cQ,\O_C)$ is a PT pair. The objects $A^\bullet$ in this locus can be viewed as some kind of ``rank 2 PT pairs''. Note that this notion a priori depends on the choice of cosection $\ccR \rightarrow \O_X$.
\item One of the key contributions of this paper is that we give a geometric description of the ``stable objects'' $A^\bullet$ of (2) purely in terms of the singularities of $\ccR$. At the level of generating functions this gets rid of the choice of cosection $\ccR \rightarrow \O_X$. Specifically, in Theorem \ref{alt}, we prove that there exists a geometric bijection between the locus of ``stable objects'' $A^\bullet$ described in (2) and $$\bigsqcup_n \Quot(\ext^1(\ccR,\O_X),n).$$  
\end{enumerate}

The RHS of Theorem \ref{main} can be seen as $M(q)^{2 e(X)}$ times the generating function of sub-PT pairs of a fixed PT pair, which we call the \emph{Hartshorne PT pair} (or rather its dual). By the \emph{Serre correspondence}, reviewed in Section \ref{Serresection}, the data of a cosection $\ccR \rightarrow \O_X$ cutting out a 1-dimensional closed subscheme is equivalent to a Hartshorne PT pair. This is discussed in Remark \ref{PTview}.

\begin{remark}\label{Toda} [Toda]
Very recently Y.~Toda \cite{Tod} developed a theory of higher rank PT pairs on any smooth projective Calabi-Yau 3-fold $X$. He proves a beautiful wall-crossing formula between the virtual counts of these pairs and higher rank Donaldson-Thomas invariants. The shape of his formula for rank 2 is similar to Theorem \ref{main}. This suggests our RHS can also be interpreted as $M(q)^{2e(X)}$ times a generating function of rank 2 PT pairs (probably with several additional constraints). It is interesting to find out the precise relationship.
\end{remark}

\noindent \textbf{Acknowledgements.} We thank R.~P.~Thomas for useful conversations. Special thanks go to B.~Young. He originally found the product formula of Theorem \ref{GKYthm} from the toric description of $\Quot(\ccR,n)^T$ using the double dimer model \cite{GKY1,GKY2}. His formula prompted this project. We thank the anonymous referee for pointing out why one needs $H^2(L) = 0$ in the Hartshorne-Serre correspondence (Theorem \ref{Serre}). We also warmly thank J.~Rennemo for pointing out that the assumptions of Theorem \ref{main} can be weakened (Remark \ref{weakassump}). A.G.~was partially supported by NSF grant DMS-1406788. M.K.~was supported by Marie Sk{\l}odowska-Curie Project 656898.

\subsection{Rennemo's argument} Assume Theorem \ref{main} holds. Then equation \eqref{maineq} holds for \emph{any} rank 2 reflexive sheaf $\ccR$ on any smooth projective threefold $X$ (Remark \ref{weakassump}). This follows from the following argument due to Rennemo.\footnote{We take full responsibility for the way in which the argument is presented.} We may assume $\ccR$ is singular, because in the locally free case \eqref{maineq} easily follows by taking a trivialization of $\ccR$. 

Let $H$ be a very ample divisor on $X$. Then equation \eqref{maineq} for $\ccR$ is equivalent to equation \eqref{maineq} for $\ccR':=\ccR(-mH)$, for any $m$. It suffices to show that for $m \gg 0$, $\ccR'$ satisfies
\begin{itemize}
\item $H^1(\det \ccR')= H^2(\det \ccR')= 0$, and 
\item there exists a cosection $\ccR' \rightarrow \O_X$ cutting out a 1-dimensional closed subscheme.
\end{itemize}
Since $\det \ccR' = (\det \ccR)(-2mH)$, the first follows for $m \gg 0$ by Serre duality and Serre vanishing. Also $\ccR^{\prime *} = \ccR^*(mH)$ is globally generated for $m \gg 0$:
\begin{equation} \label{globgen}
\pi : V \otimes \O_{X} \twoheadrightarrow \ccR^{\prime *},
\end{equation}
where $V:=  H^0(X,\ccR^{\prime *})$. Let $X^\circ = X \setminus \mathrm{Supp} \, \ext^1(\ccR',\O_X)$, i.e.~the complement of the finitely many points where $\ccR'$ is singular, and let $N:=\dim V$. Restriction of \eqref{globgen} to points of $X^\circ$ induces a morphism
$$
f : X^\circ \rightarrow \mathrm{Gr},
$$
where $\mathrm{Gr}:=\mathrm{Gr}(N-2,V)$ denotes the Grassmannian of $(N-2)$-dimensional subspaces of $V$. Denote its universal quotient by $\Pi : V \otimes \O_{\mathrm{Gr}} \twoheadrightarrow \cQ$, then $f^* \Pi = \pi|_{X^\circ}$. For any non-zero section $t \in H^0(\mathrm{Gr}, V \otimes \O_{\mathrm{Gr}}) = V$  
$$
\Pi \circ t \in H^0(\mathrm{Gr},\cQ)
$$
cuts out a smooth locus $Z(\Pi \circ t)$ of codimension 2. By Kleiman-Bertini \cite[Thm.~III.10.8]{Har3}, for \emph{generic} $t$ we have:
$$
Z(\Pi \circ t) \times_{\mathrm{Gr}} X^\circ = f^* Z(\Pi \circ t) = Z(\pi|_{X^\circ} \circ f^* t) \subset X^\circ
$$
is empty or has codimension 2. The empty case can be ruled out.\footnote{In the empty case, we get a surjective morphism $\ccR'|_{X^\circ} \twoheadrightarrow \O_{X^\circ}$. Since $\ccR'$ is reflexive, it lifts to a surjection $\ccR' \twoheadrightarrow \O_X$. Its kernel is easily seen to be a line bundle by \cite[Prop.~1.1.6, 1.1.10]{HL} and hence $\ccR'$ is locally free, contradicting the assumption that $\ccR'$ is singular.} Dualizing $\pi|_{X^\circ} \circ f^* t \in H^0(X^\circ, \ccR^{\prime *})$ produces a cosection $\ccR'|_{X^\circ} \rightarrow \O_{X^\circ}$ cutting out a 1-dimensional closed subscheme. Since $\ccR'$ is reflexive, this extends to a cosection on $X$ cutting out a 1-dimensional closed subscheme.\footnote{Kleiman-Bertini also gives that the cosection can be taken to cut out a \emph{reduced} curve.}

\section{Moduli spaces}

\subsection{Serre Correspondence} \label{Serresection} We start by recalling the Serre correspondence \cite{Har1, Har2}. For a smooth projective 3-fold $X$, the Serre correspondence gives a bijection between rank 2 locally free sheaves on $X$ with section and certain lci curves on $X$ \cite{Har1}. R.~Hartshorne extended this to reflexive sheaves \cite{Har2}.
\begin{theorem} [Hartshorne-Serre correspondence] \label{Serre}
Let $X$ be a smooth projective 3-fold and $L$ a line bundle on $X$ satisfying $H^1(L) =H^2(L) =0$. Then there exists a bijective correspondence between:
\begin{itemize}
\item Pairs $(\ccR,\sigma)$, where $\ccR$ is a rank 2 reflexive sheaf on $X$ with $\det(\ccR) \cong L$ and $\sigma : \ccR \rightarrow \O_X$ a cosection cutting out a 1-dimensional closed subscheme.
\item Pairs $(C,\xi)$, where $C \subset X$ is a Cohen-Macaulay curve which is generically lci and $\xi : \O_X \rightarrow \omega_C \otimes \omega_{X}^{-1} \otimes L$ has 0-dimensional cokernel.
\end{itemize}
\end{theorem}
In this theorem, $\omega_C \otimes \omega_{X}^{-1} \otimes L$ is a 1-dimensional pure sheaf. This can be seen by noting $\omega_C \cong \ext^2(\O_C,\omega_X)$ \cite[Thm.~21.5]{Eis} and using \cite[p.~6]{HL}. Hence $(\omega_C \otimes \omega_{X}^{-1} \otimes L,\xi)$ is a PT pair in the sense of \cite{PT1}. We introduce some notation. For any sheaf $\cE$ on $X$ with support of codimension $c$ and any line bundle $M$ on $X$
\begin{equation} \label{D}
(\cE)^{D}_{M} := \ext^c(\cE,M).
\end{equation}
Then $(\cE)^D_{\omega_X}$ is the usual notion of a dual sheaf $\cE^D$ of \cite[Def.~1.1.7]{HL}. Moreover
$$
\omega_C \otimes \omega_{X}^{-1} \otimes L \cong (\O_C)^{D}_{L}.
$$

We now briefly recall how the Serre correspondence works (following \cite{Har2} which discusses the case $X = \PP^3$ but straightforwardly generalizes). Suppose we are given data $(\ccR,\sigma)$ as in (1). Then the cosection induces a surjection $\ccR \twoheadrightarrow I_C$ and a short exact sequence
\begin{equation} \label{Koszul}
0 \longrightarrow L \longrightarrow \ccR \longrightarrow I_C \longrightarrow 0,
\end{equation}
where the kernel $L \cong \det(\ccR)$ is a line bundle. See \cite{Har2} for details. Therefore we obtain an extension $\xi \in \Ext^1(I_C,L)$. Note that $\ccR$ is locally free outside a 0-dimensional closed subscheme so $C$ is generically lci. Outside the lci locus $C$ is Cohen-Macaulay. Using the short exact sequence 
\begin{equation} \label{IC}
0 \longrightarrow I_C \longrightarrow \O_X \longrightarrow \O_C \longrightarrow 0,
\end{equation}
we obtain isomorphisms
$$
\ext^1(I_C,L) \cong \ext^2(\O_C, L) \cong \omega_C \otimes \omega_{X}^{-1} \otimes L =: (\O_C)^{D}_{L},
$$
where we used definition \eqref{D}. Using the local-to-global spectral sequence and $H^1(L) =H^2(L) = 0$, we deduce
$$
\Ext^1(I_C,L) \cong H^0(\ext^1(I_C,L)) \cong H^0((\O_C)^{D}_{L}).
$$
This gives the section $\xi : \O_X \rightarrow (\O_C)^{D}_{L}$. Applying $\hom(\cdot, L)$ to \eqref{Koszul} gives
\begin{equation} \label{Hartshornelong}
0 \longrightarrow L \longrightarrow \ccR^* \otimes L \longrightarrow \O_X \stackrel{\delta}{\longrightarrow} \ext^1(I_C,L) \longrightarrow \ext^1(\ccR,L) \longrightarrow 0, 
\end{equation}
where $\delta$ sends $1$ to $\xi$. By a slight abuse of notation, we denote $\delta$ by $\xi$ as well. Finally, $\ext^1(\ccR,L)$ is a 0-dimensional sheaf which is supported on the locus where $\ccR$ is not locally free and the length of this sheaf equals $c_3(\ccR)$ (see \cite{Har2} for details). We indeed obtain a PT pair $((\O_C)^{D}_{L},\xi)$, which we refer to as a \emph{Hartshorne PT pair}. This shows how to go from (1) to (2) in Theorem \ref{Serre}. The way backwards is explained in \cite{Har2}.

\subsection{Moduli of sheaves} Our initial moduli space is $\Quot(\ccR,n)$ of the introduction. We also define $$\Quot(\ccR) := \bigsqcup_{n=0}^{\infty} \Quot(\ccR,n).$$ Set theoretically one can write $\Quot(\ccR)$ as
$$
\bigsqcup_{\cQ \in \T} \Hom(\ccR,\cQ)^{\onto}/ \sim,
$$
where $\T$ denotes the stack of all 0-dimensional sheaves on $X$, ``onto'' refers to the subset of surjective maps in 
\begin{equation} \label{extpure1}
\Hom(\ccR,\cQ) \cong \Ext^1(\ccR,\cQ[-1]),
\end{equation} 
and the equivalence $\sim$ is induced from the automorphisms of $\cQ$. In Section \ref{PTmoduli} we define the ``PT analog'' of $\Quot(\ccR)$ (\eqref{setdef} below)
$$
\bigsqcup_{\cQ \in \T} \Ext^2(\cQ,\ccR)^{\pure}/ \sim 
$$
where 
$$
\Ext^2(\cQ,\ccR) \cong \Ext^1(\cQ[-1],\ccR).
$$

\subsection{Ext groups} Although $\Ext^1(\ccR,\cQ[-1])$ and $\Ext^1(\cQ[-1], \ccR)$ can jump, their difference cannot. This is the analog of \cite[Lem.~4.10]{ST}:
\begin{lemma} \label{lem1}
For any rank 2 reflexive sheaf $\ccR$ and 0-dimensional sheaf $\cQ$ on a smooth projective 3-fold $X$, the only Ext groups between $\ccR$ and $\cQ[-1]$ are $\Ext^1(\ccR,\cQ[-1])$, $\Ext^1(\cQ[-1], \ccR)$. Moreover
$$
\dim \Ext^1(\ccR,\cQ[-1]) - \dim \Ext^1(\cQ[-1],\ccR) = \chi(\ccR,\cQ) =  2 \ell(\cQ),
$$ 
where $\ell(\cQ)$ denotes the length of $\cQ$.
\end{lemma}
\begin{proof}
Since $\ccR$ is reflexive, it has homological dimension $\leq 1$ (\cite[Prop.~1.3]{Har2} and Auslander-Buchsbaum). Hence $\Ext^{\geq 2}(\ccR,\cQ) = 0$ and
\begin{align*}
2\ell(\cQ) = \chi(\ccR,\cQ) &= \dim \Hom(\ccR,\cQ) - \dim \Ext^1(\ccR,\cQ) \\
&= \dim \Ext^1(\ccR,\cQ[-1]) - \dim \Ext^2(\cQ,\ccR) \\ 
&= \dim \Ext^1(\ccR,\cQ[-1]) - \dim \Ext^1(\cQ[-1],\ccR), 
\end{align*}
where we use Serre duality and $\cQ \otimes \omega_{X}^{-1} \cong \cQ$ since $\cQ$ is 0-dimensional.
\end{proof}

\subsection{Moduli of PT pairs} \label{PTmoduli} In this section we fix data $(\ccR,\sigma)$ or, equivalently, $((\O_C)^{D}_{L},\xi)$ as in the Serre correspondence. For $\cQ \in \T$, we want to define $\Ext^2(\cQ,\ccR)^{\pure} \subset \Ext^2(\cQ,\ccR)$. In the rank one case of Stoppa-Thomas \cite{ST}
$$
\Ext^1(\cQ[-1],I_C) \cong \Ext^2(\cQ,I_C) \cong \Ext^1(\cQ,\O_C),
$$
where the second isomorphism is induced by \eqref{IC}. Elements of $\Ext^1(\cQ,\O_C)$ are short exact sequences
$$
0 \rightarrow \O_C \rightarrow F \rightarrow \cQ \rightarrow 0,
$$
where $F$ is necessarily 1-dimensional. Define $\Ext^1(\cQ,\O_C)^{\pure}$ as the locus of extensions for which $F$ is pure. Then elements of $\Ext^1(\cQ,\O_C)^{\pure}$ are exactly the PT pairs of \cite{PT1} for which the underlying Cohen-Macaulay support curve is $C$.
\begin{lemma} \label{lem2}
Let $\ccR$ be a rank 2 reflexive sheaf on a smooth projective 3-fold $X$ and let $\sigma : \ccR \rightarrow \O_X$ be a cosection cutting out a 1-dimensional closed subscheme $C \subset X$. Then for any 0-dimensional sheaf $\cQ$ on $X$, there exists a natural inclusion
$$
\Ext^1(\cQ[-1],\ccR) \hookrightarrow \Ext^1(\cQ[-1],I_C).
$$
\end{lemma}
\begin{proof}
This follows by applying $\Hom(\cQ,\cdot)$ to the short exact sequence \eqref{Koszul}
$$
\cdots \longrightarrow \Ext^2(\cQ,L) \longrightarrow \Ext^2(\cQ,\ccR) \longrightarrow \Ext^2(\cQ,I_C) \longrightarrow \cdots,
$$
where $\Ext^2(\cQ,L) \cong \Ext^1(L,\cQ)^* \cong H^1(\cQ)^* \cong 0$ since $\cQ$ is 0-dimensional.
\end{proof}
Using the embedding of Lemma \ref{lem2}, we can view $$\Ext^2(\cQ,\ccR) \subset \Ext^1(\cQ,\O_C)$$ as a sublocus and define
\begin{equation*} 
\Ext^2(\cQ,\ccR)^{\pure} :=  \Ext^2(\cQ,\ccR) \cap \Ext^1(\cQ,\O_C)^{\pure}.
\end{equation*}
Note that this definition depends not only on $\ccR$ but \emph{also} on the chosen cosection $\sigma : \ccR \rightarrow \O_X$. At the level of sets we consider
\begin{equation} \label{setdef}
\bigsqcup_{\cQ \in \T} \Ext^2(\cQ,\ccR)^{\pure}/ \sim,
\end{equation} 
where again the equivalence $\sim$ is induced from the automorphisms of $\cQ$. By construction, this is a sublocus of the moduli space of all PT pairs. We will now prove that the points of \eqref{setdef} are in bijective correspondence with the closed points of a certain Quot scheme of a 0-dimensional sheaf. First a useful observation:
\begin{remark} \label{dualizing}
Let $\cQ$ be any 0-dimensional sheaf on $X$, $F$ any 1-dimensional sheaf on $X$, and $L$ any line bundle on $X$. Then $\cQ$ is automatically pure and $F$ is pure if and only if $\ext^3(F,\O_X) = 0$ by \cite[Prop.~1.1.10]{HL}. Suppose $F$ is pure then, again by \cite[Prop.~1.1.10]{HL}, 
\begin{equation} \label{DD}
(F^{D}_{L})^{D}_{L} \cong F, \ (\cQ^{D}_{L})^{D}_{L} \cong \cQ.
\end{equation}
Suppose we are given any PT pair
\begin{equation} \label{pair}
0 \longrightarrow \O_C \longrightarrow F \longrightarrow \cQ \longrightarrow 0.
\end{equation}
Then dualizing using \cite[Prop.~1.1.6]{HL} gives a short exact sequence
\begin{equation} \label{dualpair}
0 \longrightarrow F^{D}_{L} \longrightarrow (\O_C)^{D}_{L} \longrightarrow \cQ^{D}_{L} \longrightarrow 0,
\end{equation}
where $F^{D}_{L}$ is pure and $\cQ^{D}_{L}$ is 0-dimensional. When $F = (\O_C)^{D}_{L}$, e.g.~when $(F,s)$ is a Hartshorne PT pair, \eqref{dualpair} defines a PT pair as well by \eqref{DD}. So the dual of a Hartshorne PT pair $((\O_C)^{D}_{L},\xi)$ is a PT pair $((\O_C)^{D}_{L},\xi^{D}_{L})$. In general, dualizing \eqref{dualpair} again gives back \eqref{pair}. Hence the operation $(\cdot)^{D}_{L}$ is reversible on PT pairs and defines an involution on PT pairs with $F = (\O_C)^{D}_{L}$.
\end{remark}

\begin{theorem} \label{alt}
Let $\ccR$ be a rank 2 reflexive sheaf on a smooth projective 3-fold $X$. Suppose $H^1(\det(\ccR)) =H^2(\det(\ccR)) = 0$ and $\sigma : \ccR \rightarrow \O_X$ is a cosection cutting out a 1-dimensional closed subscheme $C \subset X$.
Then the elements of \eqref{setdef} are in natural bijective correspondence with the closed points of $\Quot(\ext^1(\ccR,\O_X))$.
\end{theorem}
\begin{proof}
\textbf{Step 1:} We are given the short exact sequence
\begin{equation} \label{given}
0 \longrightarrow L \longrightarrow \ccR \longrightarrow I_C \longrightarrow 0
\end{equation}
and the corresponding Hartshorne PT pair $((\O_C)^{D}_{L},\xi)$. Let $\cQ \in \T$. An element of $\Ext^1(\cQ[-1],\ccR)$ corresponds to an exact triangle 
$$
\ccR \rightarrow A^\mdot \rightarrow \cQ[-1].
$$
Then the image under the natural inclusion $\Ext^1(\cQ[-1],\ccR) \subset \Ext^1(\cQ[-1],I_C)$ of Lemma \ref{lem2} is the induced third column of the following diagram
\begin{displaymath}
\xymatrix
{
L \ar[r] \ar@{=}[d] & \ccR \ar[r] \ar[d] & I_C \ar@{-->}[d] \\
L \ar[r] & A^\mdot \ar[r] \ar[d] & I^\mdot \ar@{-->}[d] \\
& \cQ[-1] \ar@{=}[r] & \cQ[-1].
}
\end{displaymath}
More precisely, an element $I_C \rightarrow I^\mdot \rightarrow \cQ[-1]$ of $\Ext^1(\cQ[-1],I_C)$ lies in the image of the natural inclusion of Lemma \ref{lem2} if and only if there exists a map $I^\mdot \rightarrow L[1]$ such that the following diagram (not an exact triangle) commutes
\begin{displaymath}
\xymatrix
{
I_C \ar[r] \ar[dr] & I^\mdot \ar@{-->}^<<<<{\exists}[d] \\
& L[1],
}
\end{displaymath}
where the diagonal map comes from \eqref{given}. This can be rephrased: the element $I_C \rightarrow I^\mdot \rightarrow \cQ[-1]$ gives rise to an exact sequence $\Ext^1(I^\mdot, L) \rightarrow \Ext^1(I_C,L) \rightarrow \Ext^3(\cQ,L)$ and the above condition is equivalent to 
$$
\xi \in \mathrm{im}\Big( \Ext^1(I^\mdot, L) \rightarrow \Ext^1(I_C,L) \Big),
$$
where $\xi \in \Ext^1(I_C,L)$ is the given extension \eqref{given}. By exactness, this is equivalent to 
\begin{equation} \label{kercon} 
\xi \in \ker\Big( \Ext^1(I_C,L) \rightarrow \Ext^3(\cQ,L) \Big).
\end{equation}
Next we note that $\ext^i(I_C,L) \cong \ext^{i+1}(\O_C,L)$ for all $i>0$. Moreover, since $C$ is Cohen-Macaulay, $\ext^i(\O_C,L) = 0$ unless $i=2$. By the local-to-global spectral sequence and $H^1(L)=H^2(L)=0$ 
\begin{align*}
\Ext^1(I_C,L) &\cong H^0(\ext^1(I_C,L)), \\
\Ext^3(\cQ,L) &\cong H^0(\ext^3(\cQ,L)).
\end{align*}
Hence an element $I_C \rightarrow I^\mdot \rightarrow \cQ[-1]$ of $\Ext^1(\cQ[-1],I_C)$ lies in the image of the natural inclusion of Lemma \ref{lem2} if and only if the composition
$$
\O_X \stackrel{\xi}{\longrightarrow} \ext^1(I_C,L) \longrightarrow \ext^3(\cQ,L)
$$
is zero. Here the map $\ext^1(I_C,L) \rightarrow \ext^3(\cQ,L)$ is induced by $I_C \rightarrow I^\mdot \rightarrow \cQ[-1]$. This in turn is equivalent to the existence of a morphism
$$
\coker(\xi) = \ext^1(\ccR,L) \longrightarrow \ext^3(\cQ,L) 
$$
for which the triangle in the following diagram commutes
\begin{equation} \label{factor}
\begin{gathered}
\xymatrix
{
\O_X \ar^<<<<{\xi}[r] & \ext^1(I_C,L) \ar[dr] \ar[r] & \ext^1(\ccR,L) \ar@{-->}^<<<<{\exists}[d] \\
& & \ext^3(\cQ,L).
}
\end{gathered}
\end{equation}
Here the first line is the Hartshorne PT pair \eqref{Hartshornelong}. \\

\noindent \textbf{Step 2:} We claim that an element in the image of $$\Ext^1(\cQ[-1],\ccR) \subset \Ext^1(\cQ[-1],I_C) \cong \Ext^1(\cQ,\O_C)$$ lies in $\Ext^1(\cQ,\O_C)^{\pure}$ if and only if the arrow in \eqref{factor} is a surjection. Let $I_C \rightarrow I^\mdot \rightarrow \cQ[-1]$ be an element of $\Ext^1(\cQ[-1],I_C)$. We get an exact sequence
$$
\cdots \longrightarrow \ext^1(I_C,L) \longrightarrow \ext^3(\cQ,L) \longrightarrow \ext^2(I^\mdot, L) \longrightarrow 0,
$$
where we use $\ext^2(I_C,L) \cong \ext^3(\O_C,L) = 0$ since $C$ is Cohen-Macaulay. Moreover, writing $I^\mdot = \{\O_X \rightarrow F\}$ we have an exact triangle $I^\mdot \rightarrow \O_X \rightarrow F$. Applying $\hom(\cdot,\O_X)$ to this exact triangle induces an isomorphism
$$
\ext^2(I^\mdot,\O_X) \cong \ext^3(F,\O_X).
$$
The claim follows because $F$ is pure if and only if $\ext^3(F,\O_X) = 0$. \\

\noindent \textbf{Step 3:} We rephrase the result of Step 2. Given a PT pair $I^\mdot = \{\O_X \rightarrow F\}$, we can form the following diagram
\begin{equation} \label{Amindiag0}
\begin{gathered}
\xymatrix
{
& & 0 \ar[d] & & \\
& & \O_C \ar^>>>>>{\xi}[d] & & \\
0 \ar[r] & F^{D}_{L} \ar[r] & (\O_C)^{D}_{L} \ar[r] \ar[d] & \cQ^{D}_{L} \ar[r] \ar@{=}[d] & 0 \\
& & \cQ_\xi \ar[d] \ar@{-->>}^<<<<<{\exists}[r] & \cQ^{D}_{L}, & \\
& & 0 & &
}
\end{gathered}
\end{equation}
where the middle column is the Hartshorne PT pair and the middle row is the dual of the given PT pair. In particular $\cQ_\xi := \coker(\xi) = \ext^1(\ccR,L)$. Then $I^\mdot$ lies in $\Ext^2(\cQ,\ccR)^{\pure}$ if and only if the indicated surjection exists. The surjection produces the desired element of $\Quot(\ext^1(\ccR,L))$. Since dualizing is a reversible operation (Remark \ref{dualizing}), we obtain a bijective map to $\Quot(\ext^1(\ccR,L))$.
\end{proof}

\begin{remark} [Sub-PT pairs] \label{PTview}
One does not actually have to stop at Step 3 of the previous proof. Fix the data $(\ccR,\sigma)$ as in the Serre Correspondence. Then we are also given a Hartshorne PT pair $((\O_C)^{D}_{L},\xi)$, whose cokernel we denote by $\cQ_\xi$. As we observed in Remark \ref{dualizing}, a special feature of the Hartshorne PT pair is that its dual is also a PT pair which we denote by $((\O_C)^{D}_{L},\xi^{D}_{L})$. We introduce the ``moduli space''\footnote{Since we only deal with Euler characteristics, we only observe $P_n(C,\xi)$ is a constructible subset of the usual space of PT pairs.} $P_n(C,\xi)$
$$
\Big\{ (F,s) \ : \ (F,s) \ \mathrm{is \ a \ PT \ pair \ with \ } (F,s) \subset ((\O_C)^{D}_{L},\xi^{D}_{L}) \ \mathrm{and} \ \chi(F) = \chi(\O_C) + n \Big\},
$$
 where $(F,s) \subset ((\O_C)^{D}_{L},\xi^{D}_{L})$ means there exists an embedding $F \hookrightarrow (\O_C)^{D}_{L}$ such that the following diagram commutes
 \begin{displaymath}
 \xymatrix
 {
 \O_X \ar^{s}[r] \ar_{\xi^{D}_{L}}[dr] & F \ar@^{(->}[d] \\
 & (\O_C)^{D}_{L}.
 }
 \end{displaymath}
In words: $\bigsqcup_n P_{n}(C,\xi)$ consists of sub-PT pairs of the dual of the Hartshorne PT pair. We claim Theorem \ref{main} implies:
\begin{equation} \label{PTformula}
\sum_{n=0}^{\infty} e(\Quot(\ccR,n)) \, q^n = M(q)^{2e(X)} \sum_{n=0}^{\infty} e(P_n(C,\xi)) \, q^n.
\end{equation}
We now prove this formula. Use the notation of Step 3 of the previous proof. If we have a surjection as indicated, diagram \eqref{Amindiag0} can be completed to
\begin{equation} \label{Amindiag1}
\begin{gathered}
\xymatrix
{
& 0 \ar[d] & 0 \ar[d] & & \\
& \O_C \ar@{=}[r] \ar[d] & \O_C \ar^>>>>>{\xi}[d] & & \\
0 \ar[r] & F^{D}_{L} \ar[r] \ar[d] & (\O_C)^{D}_{L} \ar[r] \ar[d] & \cQ^{D}_{L} \ar[r] \ar@{=}[d] & 0 \\
0 \ar[r] & \S \ar[d] \ar[r] & \cQ_\xi \ar[d] \ar[r] & \cQ^{D}_{L} \ar[r] & 0 \\
& 0 & 0 & &
}
\end{gathered}
\end{equation}
where $\S$ denotes the kernel of $\cQ_\xi \twoheadrightarrow \cQ^{D}_{L}$. Dualizing diagram \eqref{Amindiag1} yields
\begin{equation} \label{Amindiag2}
\begin{gathered}
\xymatrix
{
& & 0 \ar[d] & 0 \ar[d] & & \\
& 0 \ar[r] & \O_C \ar[r] \ar^>>>>>{\xi^{D}_{L}}[d] & F \ar[r] \ar[d] & \cQ \ar[r] & 0 \\
& & (\O_C)^{D}_{L} \ar@{=}[r] \ar[d] & (\O_C)^{D}_{L} \ar[d] & & \\
0 \ar[r] & \cQ \ar[r] & (\cQ_\xi)^{D}_{L} \ar[r] \ar[d] & \S^{D}_{L} \ar[r] \ar[d] & 0 & \\
& & 0 & 0 & &
}
\end{gathered}
\end{equation}
The top square gives rise to the description of $P_n(C,\xi)$. Dualizing this diagram we can go back to \eqref{Amindiag1} and \eqref{Amindiag0} (Remark \ref{dualizing}). Once we establish Theorem \ref{main} (done in the next section), formula \eqref{PTformula} follows from the bijection
$$
\bigsqcup_{n=0}^{\infty} \Quot(\ext^1(\ccR,\O_X),n) \leftrightarrow \bigsqcup_{n=0}^{\infty} P_{n}(C,\xi).
$$
\end{remark}

\section{Hall algebra calculation}

In this section we prove Theorem \ref{main} using Theorem \ref{alt} and a rank 2 version of a Hall algebra calculation of Stoppa-Thomas \cite{ST}. As before, let $X$ be a smooth projective 3-fold and let $\ccR$ be a rank 2 reflexive sheaf on $X$ with $H^1(\det(\ccR)) =H^2(\det(\ccR)) = 0$ and a cosection $\sigma : \ccR \rightarrow \O_X$ cutting out a 1-dimensional closed subscheme. We denote by $\T$ the stack of 0-dimensional sheaves on $X$. We use the following $\T$-stacks:
\begin{itemize}
\item $1_\T$ is the identity map $\T \rightarrow \T$,
\item $\Hom(\ccR, \cdot)$ is the stack whose fibre over $\cQ \in \T$ is $\Hom(\ccR, \cQ)$,
\item $\Hom(\ccR, \cdot)^{\onto}$ is the stack whose fibre over $\cQ \in \T$ is $\Hom(\ccR, \cQ)^{\onto}$,
\item $\Ext^2(\cdot,\ccR)$ is the stack whose fibre over $\cQ \in \T$ is $\Ext^2(\cQ, \ccR)$,
\item $\Ext^2(\cdot,\ccR)^{\pure}$ is the stack whose fibre over $\cQ \in \T$ is 
$$
\Ext^2(\cQ, \ccR)^{\pure} = \Ext^2(\cQ, \ccR) \cap \Ext^1(\cQ, \O_C)^{\pure} \subset \Ext^1(\cQ, \O_C),
$$
where we use the embedding $\Ext^2(\cQ,\ccR) \subset \Ext^1(\cQ,\O_C)$ of Lemma \ref{lem2}.
\end{itemize}

Denote by $H(\T) := K(\mathrm{St}/\T)$ the Grothendieck group of stacks (locally of finite type and with affine geometric stabilizers) over $\T$. Then $H(\T)$ can be endowed with the following product. Let $\T^2$ be the stack of short exact sequences $0 \rightarrow \cQ_1 \rightarrow \cQ \rightarrow \cQ_2 \rightarrow 0$ in $\T$ and let $\pi_i$ be the map induced by sending this short exact sequence to $\cQ_i$. For any two ($\T$-isomorphism classes of) $\T$-stacks $[U \rightarrow \T]$ and $[V \rightarrow \T]$, the product $[U * V \rightarrow \T]$ is defined by the following Cartesian diagram
\begin{equation}
\xymatrix
{
U*V \ar[d] \ar[r] & \T^2 \ar[d] \ar[r] & \T \\
U \times V \ar[r] & \T \times \T.
}
\end{equation}
This makes $(H(\T),*)$ into an associative algebra, known as a Joyce's motivic Ringel-Hall algebra. See \cite{ST} and \cite{Bri1,Bri2,Joy1,Joy2,Joy3,Joy4,KS} for details.

Using the inclusion-exclusion principle, we can write $\Hom(\ccR,\cQ)^{\onto}$ as
$$
\Hom(\ccR,\cQ) - \bigsqcup_{\cQ_1 < \cQ} \Hom(\ccR,\cQ_1) + \bigsqcup_{\cQ_1 < \cQ_2 <\cQ} \Hom(\ccR,\cQ_1) - \cdots,
$$
where $<$ denotes \emph{strict} inclusion. This leads to the analog in our context of Bridgeland's generalization of Reineke's formula \cite{Bri2} 
$$
\Hom(\ccR,\cdot) = \Hom(\ccR,\cdot)^{\onto} * 1_\T.
$$
We recall the following identity from \cite{ST}
$$
\Ext^1(\cdot,\O_C) = 1_\T * \Ext^1(\cdot, \O_C)^{\pure}.
$$
Since $\Ext^2(\cdot,\ccR) \subset \Ext^1(\cdot,\O_C)$ (Lemma \ref{lem2}), ``intersecting'' this equation with $\Ext^2(\cdot,\ccR)$ yields\footnote{Given $I^\mdot = \{\O_X \rightarrow F \} \in \Ext^1(\cdot,\O_C)$, modding out by the torsion subsheaf of $F$ leads to an element of $1_{\T} * \Ext^1(\cdot, \O_C)^{\pure}$. This map is a geometric bijection. In our context one has to verify that if $I^\mdot$ has property \eqref{kercon}, then the associated PT pair $\tilde{I}^{\mdot} = \{\O_X \rightarrow \tilde{F} \}$, where $\tilde{F} = F / \mathrm{torsion}$ also has property \eqref{kercon}. This follows from the fact that the natural map $\ext^3(\tilde{\cQ},L) \rightarrow \ext^3(\cQ,L)$ is injective, where $\cQ$, $\tilde{\cQ}$ are the cokernels of $I^{\mdot}$, $\tilde{I}^{\mdot}$ respectively.}
$$
\Ext^2(\cdot,\ccR) = 1_\T * \Ext^2(\cdot, \ccR)^{\pure}.
$$

Now let 
$$
P_z(\cdot) : H(\T) \longrightarrow \Q(z)[\![q]\!],
$$
denote the virtual Poincar\'e polynomial. Here $z$ is the formal variable of $P_z$ and $q$ keeps track of an additional grading as follows. Any element $[U \rightarrow \T] \in H(\T)$ is locally of finite type and can have infinitely many components. Let $\T_n \subset \T$ be the substack of 0-dimensional sheaves of length $n$ and define
$$
P_z(U) := \sum_{n=0}^{\infty} P_z(U \times_{\T} \T_n) \, q^n.
$$
In fact, $P_z(\cdot)$ is a \emph{Lie algebra homomorphism} to the abelian Lie algebra $\Q(z)[\![q]\!]$ \cite[Thm.~4.32]{ST}. We obtain
\begin{align*}
P_z \big(\Ext^2(\cdot,\ccR)^{\pure} \big)(z^2 q) &= P_z \big(1_{\T}^{-1} * \Ext^2(\cdot,\ccR) \big)(z^2 q) \\
&= P_z \big(\Ext^2(\cdot,\ccR) * 1_{\T}^{-1} \big)(z^2 q), 
\end{align*}
where we weigh the component over $\T_n$ by $z^2q$ instead of $q$ for reasons that will become apparent. We need one more $\T$-stack: $\C^{2\ell(\cdot)}$, whose fibre over $\cQ \in \T$ is $\C^{2\ell(\cQ)}$, where $\ell(\cQ)$ denotes the length of $\cQ$. Over strata in $\T$ where $\Hom(\ccR,\cdot)$ is constant, the stacks 
$$
\Hom(\ccR,\cdot), \ \Ext^2(\cdot,\ccR) \oplus \C^{2\ell(\cdot)}
$$
are both Zariski locally trivial of the same rank by Lemma \ref{lem1}. We deduce
\begin{align*}
P_z \big(\Ext^2(\cdot,\ccR) * 1_{\T}^{-1} \big)(z^2 q) &=P_z \Big( \big(\Ext^2(\cdot,\ccR) \oplus \C^{2\ell(\cdot)}\big) *  \big(\C^{2\ell(\cdot)} \big)^{-1} \Big)(q) \\
&= P_z\big( \Hom(\ccR,\cdot) * (\C^{2\ell(\cdot)})^{-1} \big)(q).
\end{align*}
On the other hand
\begin{align*}
P_z\big(\Hom(\ccR, \cdot)^{\onto} \big)(q) &= P_z \big(\Hom(\ccR, \cdot) * 1_{\T}^{-1} \big)(q) \\
&= P_z \Big( \big(\Hom(\ccR, \cdot) * (\C^{2\ell(\cdot)})^{-1} \big) * \big(\C^{2\ell(\cdot)} * 1_{\T}^{-1}\big) \Big)(q). 
\end{align*}
Define $U := \Hom(\ccR, \cdot) * (\C^{2\ell(\cdot)})^{-1}$ and $V:=\C^{2\ell(\cdot)} * 1_{\T}^{-1}$. By \cite[Thm.~4.34]{ST}, if both $\lim_{z \rightarrow 1} P_z(U)$ and $\lim_{z \rightarrow 1} P_z(V)$ exist, we have
$$
\lim_{z \rightarrow 1} P_z(U*V) = \lim_{z \rightarrow 1} P_z(U) \, \lim_{z \rightarrow 1} P_z(V).
$$
In our setting
\begin{align*}
&\lim_{z \rightarrow 1} P_z \big(\C^{2\ell(\cdot)} * 1_{\T}^{-1} \big)(q) = \lim_{z \rightarrow 1} P_z \big(\Hom(\O_{X}^{\oplus 2}, \cdot)^{\onto} \big)(q) = M(q)^{2e(X)}, \\
&\lim_{z \rightarrow 1} P_z \big(\Hom(\ccR, \cdot) * (\C^{2\ell(\cdot)})^{-1} \big)(q) = \lim_{z \rightarrow 1} P_z \big(\Ext^2(\cdot,\ccR)^{\pure} \big)(z^2q) \\
&\qquad\qquad\qquad\qquad\qquad\qquad\qquad \ \ = \sum_{n=0}^{\infty} e \big( \Quot(\ext^1(\ccR,\O_X),n) \big) \, q^n, \\
&\lim_{z \rightarrow 1} P_z(U*V) = \lim_{z \rightarrow 1}P_z\big(\Hom(\ccR, \cdot)^{\onto}\big)(q) = \sum_{n=0}^{\infty} e\big(\Quot(\ccR,n)\big) \, q^n,
\end{align*}
where the fourth equality follows from Theorem \ref{alt}. This proves Theorem \ref{main}.

\section{Toric calculation} \label{torus}

In this section we prove Theorem \ref{GKYthm} by combining Theorem \ref{main} with $T$-localization. Since Theorem \ref{main} is about smooth \emph{projective} 3-folds, we start by compactifying $\C^3$. Let $X = \PP^3$ with homogeneous coordinates $X_0,X_1,X_2,X_3$ and let $U_0 = \{X_0 \neq 0\} \cong \C^3$.

In chart $U_0$, we use coordinates $x_i = X_i / X_0$, for $i=1,2,3$. The action of $T$ on $U_0$ is given by $t \cdot x_i = t_i x_i$. We describe a \emph{singular} rank 2 $T$-equivariant reflexive sheaf on $\C^3$ by describing the corresponding $\Z^3$-graded $\C[x_1,x_2,x_3]$-module. Here $X(T) = \Z^3$ is the character group of $T$. For each $w \in \Z^3$ consider the vector space $\C \cdot \bfone_w \oplus \C \cdot \bftwo_w \oplus \C \cdot \bfthree_w$, where $\bfone_w, \bftwo_w, \bfthree_w$ denote independent vectors. Choose integers $u_1, u_2, u_3 \in \Z$, $v_1, v_2, v_3 \in \Z_{>0}$. The module we will construct has three homogeneous generators of weights
\begin{equation} \label{wghts2}
(u_1,u_2,u_3) + (v_1,v_2,0), \ (u_1,u_2,u_3) + (v_1,0,v_3), \ (u_1,u_2,u_3) + (0,v_2,v_3).
\end{equation}
Define
$$
A:= \bigoplus_{{\scriptsize{ \begin{array}{c} w_1 \geq u_1 + v_1 \\ w_2 \geq u_2 + v_2 \\ w_3 \geq u_3 \end{array}}}} \C \cdot \bfone_w \oplus \bigoplus_{{\scriptsize{ \begin{array}{c} w_1 \geq u_1 + v_1 \\ w_2 \geq u_2 \\ w_3 \geq u_3 + v_3 \end{array}}}} \C \cdot \bftwo_w \oplus \bigoplus_{{\scriptsize{ \begin{array}{c} w_1 \geq u_1 \\ w_2 \geq u_2 + v_2 \\ w_3 \geq u_3 + v_3 \end{array}}}} \C \cdot \bfthree_w,
$$
which is a module under $x_1 \cdot \mathbf{i}_{w} = \mathbf{i}_{w+(1,0,0)}$, $x_2 \cdot \mathbf{i}_{w} = \mathbf{i}_{w+(0,1,0)}$, $x_3 \cdot \mathbf{i}_{w} = \mathbf{i}_{w+(0,0,1)}$ (adopting notation similar to \cite{PT2}). Define the following $\Z^3$-graded $\C[x_1,x_2,x_3]$-module
$$
R_0 := A  \Big/ \bigoplus_{{\scriptsize{ \begin{array}{c} w_1 \geq u_1 + v_1 \\ w_2 \geq u_2 + v_2 \\ w_3 \geq u_3 + v_3 \end{array}}}} \C \cdot (1,1,1)_w.
$$
Then $R_0$ corresponds to a singular rank 2 $T$-equivariant reflexive sheaf $\ccR_0$ on $U_0 \cong \C^3$ generated by three homogeneous generators of weights \eqref{wghts2}. 

Rank 2 $T$-equivariant reflexive sheaves on toric 3-folds and their moduli were studied by the authors in \cite{GK1} based on the work of A.~Klyachko \cite{Kly1, Kly2}. From the toric description (\cite{Kly1,Kly2} or \cite{GK1}) it is immediate that all singular rank 2 $T$-equivariant reflexive sheaves on $\C^3$ arise in the above way (up to $T$-equivariant isomorphism). Moreover, $\ccR_0$ as above extends to a rank 2 $T$-equivariant reflexive sheaf $\ccR$ on $X = \PP^3$, which is locally free outside $[1:0:0:0]$. This ``extension'' is unique up to tensoring by a $T$-equivariant line bundle on $X$. Since $\ccR|_{U_i} \cong \O_{U_i} \oplus \O_{U_i}$ for $i=1,2,3$, we find
\begin{align*}
&\sum_{n=0}^{\infty} e\big(\Quot(\ccR,n)\big) \, q^n = \sum_{n=0}^{\infty} e\big(\Quot(\ccR,n)^T\big) \, q^n = \prod_{i=0}^{3} \sum_{n=0}^{\infty} e\big(\Quot(\ccR|_{U_i},n)^T\big) \, q^n \\
&= \Bigg( \sum_{n=0}^{\infty} e\big(\Quot(\ccR_0,n)^T\big) \, q^n \Bigg) \cdot \Bigg( \sum_{n=0}^{\infty} e\big(\Quot(\O_{\C^3}^{\oplus 2},n)^T\big) \, q^n \Bigg)^3 \\
&=M(q)^6 \sum_{n=0}^{\infty} e\big(\Quot(\ccR_0,n)^T\big) \, q^n,
\end{align*}
where the first equality is $T$-localization, the second equality follows from the fact that the cokernels are 0-dimensional, and the fourth equality follows by using the extra scaling $\C^*$-action on $\O_{U_i} \oplus \O_{U_i}$. Therefore in order to prove Theorem \ref{GKYthm} from Theorem \ref{main}, we must construct a cosection on $\ccR$ and prove
$$
\sum_{n=0}^{\infty} e\big( \Quot(\ext^1(\ccR,\O_X),n) \big) \, q^n = M(q)^8 \prod_{i=1}^{v_1} \prod_{j=1}^{v_2} \prod_{k=1}^{v_3} \frac{1-q^{i+j+k-1}}{1-q^{i+j+k-2}}.
$$

On $U_0$ we choose a generic $T$-equivariant sub-line bundle $L_0 \subset R_0$ 
$$
L_0 := \bigoplus_{{\scriptsize{ \begin{array}{c} w_1 \geq u_1 + v_1 \\ w_2 \geq u_2 + v_2 \\ w_3 \geq u_3 + v_3 \end{array}}}} \Bigg( \C \cdot (a,b,c)_w \oplus  \C \cdot (1,1,1)_w \Big/  \C \cdot (1,1,1)_w \Bigg),
$$
where $(a,b,c)$ is sufficiently generic, e.g.~$(1,1,0)$ will do. Since the isomorphism type of $\Quot(\ccR,n)$ does not change when we replace $\ccR$ by $\ccR \otimes M$, for any $T$-equivariant line bundle $M$ on $\C^3$, we may assume without loss of generality that $u_1,u_2,u_3 = 0$ in \eqref{wghts2}. Recall that $v_1,v_2,v_3>0$ remain arbitrary. After this normalization, we obtain a $T$-equivariant isomorphism with an ideal sheaf
\begin{equation} \label{localses}
R_0 / L_0 \cong I_{C_0} \cong (x_{1}^{v_1} x_{2}^{v_2},x_{1}^{v_1} x_{3}^{v_3},x_{2}^{v_2} x_{3}^{v_3}),
\end{equation}
where $C_0 \subset U_0$ is the union of thickened coordinate axes. Again possibly after replacing $\ccR$ by $\ccR \otimes M$, one can easily show from the toric description (\cite{Kly1, Kly2} or \cite{GK1}) that $L_0$ extends to a global $T$-equivariant sub-line bundle $L \subset \ccR$, for which there exists a $T$-equivariant short exact sequence
$$
0 \longrightarrow L \longrightarrow \ccR \longrightarrow I_C \longrightarrow 0,
$$
where $C$ is the union of thickenings of the three toric $\PP^1$'s meeting in the point $[1:0:0:0]$. Therefore we may apply Theorem \ref{main}.

Finally we observe we are actually in the setting of Remark \ref{Hilbremark}. From the way we defined $R_0$, it is clear that there exists a $\Z^3$-graded short exact sequence
\begin{equation} \label{sesR}
0 \longrightarrow L_0 \longrightarrow N_0^{(1)} \oplus N_0^{(2)} \oplus N_0^{(3)} \longrightarrow R_0 \longrightarrow 0,
\end{equation}
where 
\begin{align*}
N_0^{(1)} := \bigoplus_{{\scriptsize{ \begin{array}{c} w_1 \geq v_1 \\ w_2 \geq  v_2 \\ w_3 \geq 0 \end{array}}}} \C \cdot \bfone_{w} , \ N_0^{(2)} := \bigoplus_{{\scriptsize{ \begin{array}{c} w_1 \geq v_1 \\ w_2 \geq  0 \\ w_3 \geq v_3 \end{array}}}} \C \cdot \bftwo_{w}, \ N_0^{(3)} := \bigoplus_{{\scriptsize{ \begin{array}{c} w_1 \geq 0 \\ w_2 \geq  v_2 \\ w_3 \geq v_3 \end{array}}}} \C \cdot \bfthree_{w}.
\end{align*}
Short exact sequence \eqref{sesR} easily extends to a $T$-equivariant 2-term resolution of $\ccR$ involving (sums of) line bundles. Taking $\hom(\cdot, \O_X)$ and using the local description reveals that $\ext^1(\ccR,L)$ is the structure sheaf of the 0-dimensional scheme defined by the ideal
$$
I_{\Sing(\ccR)} = (x_{1}^{v_1},x_{2}^{v_2},x_{3}^{v_3}).
$$
By \eqref{Hilb}, 
\begin{align*}
\sum_{n=0}^{\infty} e\big( \Quot(\ext^1(\ccR,\O_X),n) \big) \, q^n &= \sum_{n=0}^{\infty} e\big( \Hilb^n(\Sing(\ccR)) \big) \, q^n \\
&= \sum_{n=0}^{\infty} e\big( \Hilb^n(\Sing(\ccR))^T \big) \, q^n.
\end{align*}
The $T$-fixed points of $\Hilb^n(\Sing(\ccR))$ are in bijective correspondence with 3D partitions inside the box $[0,v_1] \times [0,v_2] \times [0,v_3]$. The generating function for these can be found in R.~P.~Stanley's book \cite[(7.109)]{Sta}. Theorem \ref{GKYthm} follows.

{\tt{amingh@umd.edu, m.kool1@uu.nl}}
\end{document}